\documentclass[twoside,11pt,reqno]{amsart}
\usepackage{amscd, bbm}
\usepackage{amssymb,mathabx}
\usepackage{todonotes}
\usepackage{mathrsfs,wasysym, amscd}
\usetikzlibrary{matrix, arrows}

\makeatletter

\@addtoreset{equation}{section}

\newtheorem{Theorem}{Theorem}[section]
\newtheorem{Lemma}[Theorem]{Lemma}
\newtheorem{Proposition}[Theorem]{Proposition}
\newtheorem{Corollary}[Theorem]{Corollary}

\newtheorem*{Theoremm}{Theorem}

\theoremstyle{remark}
\newtheorem{Remark}[Theorem]{Remark}

\newbox\squ  
\setbox\squ=\hbox{\vrule width.6pt
   \vbox{\hrule height.6pt width.4em\kern1ex
         \hrule height.6pt}%
   \vrule width.6pt}

\def\Lie{\operatorname{Lie}}

\def\Hom{\operatorname{Hom}}
\def\reg{\operatorname{reg}}

\def\ad{{\operatorname{ad}}}
\def\Ad{{\operatorname{Ad}}}

\def\End{{\operatorname{End}}}
\def\Char{{\operatorname{Char}}}

\def\opp{{\operatorname{opp}}}

\def\Mat{{\operatorname{Mat}}}
\def\rank{{\operatorname{rank}}}
\def\reg{{\operatorname{reg}}}
\def\codim{{\operatorname{codim}}}
\def\Ind{{\operatorname{Ind}}}

\def\Ker{{\operatorname{Ker}}}

\def\Os{{\O_{\operatorname{sr}}}}

\def\C{{\mathbb C}}

\def\Z{{\mathbb Z}}

\def\k{{\mathbbm k}}

\def\O{{\mathcal O}}
\def\N{{\mathcal N}}
\def\D{{\mathcal D}}
\def\0{{\bar 0}}
\def\1{{\bar 1}}

\newcommand{\g}{\mathfrak{g}}

\newcommand{\z}{\mathfrak{z}}

\renewcommand{\l}{\mathfrak{l}}
\def\p{\mathfrak{p}}

\def\g{{\mathfrak g}}

\def\n{\mathfrak n}
\def\p{\mathfrak p}

\def\ty{\tilde y}

\def\bi{\text{\sf {\bf{i}}}}

\newdimen\Hoogte    \Hoogte=10pt    
\newdimen\Breedte   \Breedte=10pt   
\newdimen\Dikte     \Dikte=0.5pt    

\title{Harish-Chandra invariants and the centre of the reduced enveloping algebra}
\author{Lewis W. Topley}

\address{Dipartimento di Matematica, Universita degli Studi di Padova, via Trieste 63, 35121 Padova, Italia}
\email{lewis@math.unipd.it}

\begin{document}

\maketitle

\begin{abstract}
In this article we consider the centre of the reduced enveloping algebra of the Lie algebra of a reductive algebraic group in very good characteristic $p > 2$. 
The Harish-Chandra centre maps to the centre of each reduced enveloping algebra and, using a combination of induction and deformation arguments, we describe
precisely for which $p$-characters this map is surjective: it is if and only if the chosen character is regular. This provides the converse to a theorem of
Mirkovi\'{c} and Rumynin.
\end{abstract}


\section{Introduction}

Throughout we take $\k$ to be an algebraically closed field of characteristic $p > 2$ and $G$ a connected reductive algebraic group of rank $\ell$ with
simply connected derived subgroup. We also assume that $p$ is very good for $G$.
The Lie algebra $\g$ is restricted in a natural way and the representation theory of $\g$ is governed by that of the
reduced enveloping algebras $U_\chi(\g)$ with $\chi \in \g^*$. When the stabaliser of $\chi$ in $\g$ is a torus the reduced enveloping algebra is semisimple whilst,
at the other end of the spectrum, the simple $U_0(\g)$-modules are precisely the differentials of simple $G$-modules with restricted highest weights.

Since $U_\chi(\g)$ plays a key role in the representation theory of $\g$ it seems important to understand the centre $Z_\chi(\g) := Z(U_\chi(\g))$
and here we make some progress towards that end. For $\chi \in \g^\ast$ we consider the map
$$\varphi_\chi : Z(\g) \longrightarrow Z_\chi(\g)$$
obtained by restricting the projection $U(\g) \twoheadrightarrow U_\chi(\g)$. Recall that the centre $Z(\g)$ of the enveloping
algebra $U(\g)$ is well understood, thanks to the work of several authors. We know that $Z(\g)$ is generated by the $p$-centre $Z_p(\g)$ and the Harish-Chandra invariants
$U(\g)^G$, and the latter has a similar description in characteristic zero and very good positive characteristic. This leads immediately to the question,
when is $\varphi_\chi$ surjective? The two extremal cases of this question are known: when $\chi$ is regular a theorem of Mirkovi\'{c} and Rumynin says that
$\varphi_\chi$ surjects \cite[Theorem~12]{MR}, whilst in the case $\chi = 0$ the map $\varphi_\chi$ does not surject, by an argument due to Premet \cite[\S 3.17]{BG}.
Apart from this very little is known about the cokernel of $\varphi_\chi$ as a linear map, and the goal of this article is to describe precisely when it is zero.

We remind the reader that an element $\chi \in \g^*$ is called regular if the stabaliser $\g_\chi := \{x\in \g : \chi [x, \g] = 0\}$ has the minimal possible dimension $\rank(\g)$. 
The following is our main result. 
\begin{Theoremm}
$Z_\chi(\g) = \varphi_\chi (Z(\g))$ if and only if $\chi$ is a regular element.
\end{Theoremm}
The idea behind the proof is quite simple. To start with we observe that isomorphism type and hence the dimension of the centre of $U_\chi(\g)$ only depends upon
the orbit $\Ad^*(G)\chi$. In the base case $G = SL_2$ or $\mathbb{G}_m$
and coadjoint orbits are either regular or trivial. In these cases the theorem follows from previously established results. When $\ell > 1$, once again the regular orbits
are dealt with by \cite{MR}, whilst every non-regular orbit $\O \subseteq \g^*$ lies in the closure of a non-nilpotent subregular decomposition class. If $\chi$ lies in 
such a decomposition class we show that $Z_\chi(\g) \cong Z_\chi(\g_{\chi_s})$  for $\chi \in \O$ and apply an inductive argument, whilst if $\chi$ lies in the boundary
of this decomposition class we apply a deformation argument to show that the dimension of the centre is larger than $\dim \varphi_\chi(Z(\g))$.

The paper is organised as follows. In Section~\ref{notations} we fix our notation and recall the elements of the theory. We then explain that $\dim \varphi_\chi(Z(\g)) = p^\ell$
and go on to recall an important category equivalence due to Kac--Weisfeiler which allows us
to relate the centre $Z_\chi(\g)$ to $Z_\chi(\g_{\chi_s})$. In Section~\ref{extremal} we recall the two known cases of the theorem and in Section~\ref{deformation}
we recall the theory of decomposition classes. It is here that we show that $\dim Z_\chi(\g)$ does not change as we vary $\chi$ over such a class and prove
a useful result which states that $\g^\ast \rightarrow \Z_{\geq 0}; \chi \mapsto \dim(Z_\chi(\g))$ is upper semicontinuous. In the final section we combine the
ingredients and complete the proof using induction on the reductive rank of $G$.

\noindent {\bf Acknowledgement.}
I would like to thank James Humphreys, David Stewart and Rudolf Tange for reading a preliminary version of this article
and making suggestions to improve the exposition. I would also like to thank Giovanna Carnovale for useful discussions
about reduced enveloping algebras and related topics. Whilst carrying out the research that led to these results the
author benefited from funding from the European Commission, Seventh Framework Programme, under Grant Agreement number
600376, as well as grants CPDA125818/12 and 60A01-4222/15 from the University of Padova.

\section{Notations and preliminaries}\label{notations}

Our notations and assumptions are the following:
\begin{enumerate}
\item{$G$ is a reductive algebraic group over $\k$;}
\smallskip
\item{$\k$ is algebraically closed of characteristic $p > 2$ and $p$ is very good for $G$;}
\smallskip
\item{$G$ is connected and the derived subgroup $[G,G]$ is simply connected.}
\smallskip
\end{enumerate}
We warn the reader that not every Levi subgroup satisfies these same hypotheses. In particular $p$ will not necessarily
be a very good prime for every Levi subgroup. Nonetheless, Lie algebras of Levi subgroups are precisely the centralisers
of semisimple elements (see \cite[2.2(4)]{Ja2} and \cite[Lemma~2.1.2]{CM}).

We always write $\g = \Lie(G)$,
write $U(\g)$ for the enveloping algebra and $S(\g)$ for the symmetric algebra.
We shall denote the centre of $\g$ by $\z(\g)$ and we observe that, since the characteristic of
$\k$ is very good for $G$, we have $\g = \z(\g) \bigoplus [\g, \g]$ and, more generally, $\l = \z(\l) \bigoplus [\l, \l]$
for every Levi subalgebra $\l = \Lie(L) \subseteq \g$ such that $p$ is very good for $L$ (see \cite[\S 2.1]{PS} for slightly more detail).

The maximal ideal of the $p$-centre corresponding to $\chi \in \g^*$ shall be written
$I_\chi := (x^p - x^{[p]} - \chi(x)^p : x\in \g)$. Then the reduced enveloping algebra is $$U_\chi(\g) := U(\g)/ I_\chi U(\g)$$
and it is customary to write $U^{[p]}(\g) = U_0(\g)$. When $\g_0\subseteq \g$ is a restricted
subalgebra, we shall abuse notation by writing
$U_\chi(\g_0)$ to denote $U_{\chi|_{\g_0}}(\g_0)$, which we view as a subalgebra of $U_\chi(\g)$.
Before we proceed we shall need an elementary fact.
\begin{Lemma}\label{supportlemma}
For each $\chi \in \g^\ast$ we may define $\chi'\in \g^\ast$ by
$$\chi'(x) := \left\{ \begin{array}{cc} \chi(x) & x\in [\g,\g] \\ 0 & x \in \z(\g) \end{array} \right. .$$
Then $U_\chi(\g) \cong U_{\chi'}(\g)$.
\end{Lemma}
\begin{proof}
Our assumption (2) and (3) imply that $\z(\g) = \Lie(Z(G))$ and $[\g,\g] = \Lie[G,G]$ (combine \cite[2.5.12]{Le} and \cite[2.1]{PS}),
hence the $p$-mapping stabalises both $\z(\g)$ and $[\g,\g]$, and it follows that
$U_\chi(\g) \cong  U_\chi(\z(\g))\bigotimes_\k U_\chi([\g,\g])$.
Now it suffices to prove that $U_\chi(\z(\g))\cong U^{[p]}(\z(\g))$. As the $p$-mapping is non-degenerate on $\z(\g)$ there is a basis
$z_1,...,z_d$ fixed by $x \mapsto x^{[p]}$ (see \cite[Theorem~3.6(1)]{FS}).
This gives $U_\chi(\z(\g)) \cong \k[z_1,...,z_d]/(z_i^p - z_i - \chi(z_i)^p : i=1,...,d)$, and similar for $U^{[p]}(\z(\g))$.
Since the polynomial $X^p - X - \lambda$ is separable for every $\lambda \in \k$
in the polynomial ring $\k[X]$ in one variable, it follows that both $U_\chi(\z(\g))$ and $U^{[p]}(\z(\g))$
are finite dimensional commutative semisimple algebras, hence isomorphic to the product of $p^d$ copies of $\k$.
\end{proof}

We shall need a description of the centre $Z(\g)$ which has been developed by many authors over the years. The first contribution was made
by Veldkamp, and successive authors have improved the restrictions on both $G$ and $p$. We direct the reader to the list of references given
in \cite[\S 9]{Ja1}. The version of the theorem stated here appeared in \cite[Theorem 2]{MR}.
\begin{Theorem}
The centre $Z(\g)$ is generated by $Z_p(\g)$ and $U(\g)^G$ as a $\k$-algebra. Furthermore, $Z(\g)$ is a free $Z_p(\g)$-module of rank $p^\ell$.
\end{Theorem}
\begin{Corollary}\label{firstcor}
$\dim \varphi_\chi Z(\g) = p^\ell$. $\hfill \qed$
\end{Corollary}

Since we are in very good characteristic there exists a rational representation of $G$ such that the associated trace form on $\g$
is non-degenerate \cite[2.5.12]{Le}. This induces an isomorphism $\kappa : \g \rightarrow \g^*$ of $G$-modules, which we keep fixed
throughout this paper. The isomorphism allows us to transfer the Jordan decomposition theorem from $\g$ to $\g^*$. For $\chi \in \g^*$ we shall write
$\chi = \chi_n + \chi_s$ for the decomposition into nilpotent and semisimple parts. The next theorem was originally due to
Kac--Weisfeiler \cite{KW}, but was later augmented to the form presented here by Friedlander--Parshall \cite[Theorem 3.2]{FP} and by Premet
\cite[Proposition~2.5]{PrST}.
\begin{Theorem}\label{KWtheorem}
Suppose that $\chi_s + \chi_n$ is the Jordan decomposition of $\chi\in \g^\ast$. Write $d = \frac{1}{2}(\dim \g - \dim \g_{\chi_s})$.
Then as algebras we have
$$U_\chi(\g) \cong \Mat_{p^d} U_\chi(\g_{\chi_s}).$$
\end{Theorem}
\begin{Corollary}\label{dimscorol}
Let $\chi \in \g^*$ and suppose that $p$ is very good for the Levi subalgebra $\g_{\chi_s}$.
Then $\dim Z_\chi(\g) = \dim Z_{\chi_n}(\g_{\chi_s})$.
\end{Corollary}
\begin{proof}
The previous theorem tells us that $\dim Z_\chi(\g) = \dim Z_{\chi}(\g_{\chi_s})$. Since the characteristic
of the field is very good for $\g_{\chi_s}$ we have $\g_{\chi_s} = \z(\g_{\chi_s}) \bigoplus [\g_{\chi_s}, \g_{\chi_s}]$ and
so we may apply Lemma~\ref{supportlemma}.
\end{proof}

\section{The cases $\chi = 0$ and $\chi$ regular}\label{extremal}

As we mentioned previously the case $\chi = 0$ of the current theorem is an example originally
devised by Premet and presented by Brown and Gordon in \cite[\S 3.17]{BG}.
\begin{Theorem}\label{zerocase}
The map $\varphi_0$ is not surjective.
\end{Theorem}
\begin{Remark}
Even after the current article, the construction underlying the proof of this theorem remains the only known example of a central element of some $Z_\chi(\g)$ which does not lie
in the image of $\varphi_\chi$. Apart from this such elements remain shrouded in mystery, although a theorem of M\"{u}ller \cite{Mu}
implies that the primitive central idempotents of $U_\chi(\g)$ all lie in the image of $\varphi_\chi$
(see also \cite[2.10]{BG}).
\end{Remark}

The regular case follows from \cite[Theorem~12]{MR}.
\begin{Theorem}
Write $d = \frac{1}{2}(\dim \g - \ell)$. For regular $\chi\in \g^\ast$ we have an isomorphism $$U_\chi(\g) \cong \Mat_{p^d} Z_\chi(\g).$$
\end{Theorem}
\begin{Corollary}\label{regularcase}
When $\chi \in \g^*$ is regular the map $\varphi_\chi$ is surjective.
\end{Corollary}
\begin{proof}
The equality $\dim U_\chi(\g) = p^{\dim(\g)}$ implies that $\dim Z_\chi(\g) = p^\ell$.
Now the conclusion follows from Corollary~\ref{firstcor}
\end{proof}
\begin{Remark}
The above statement was generalised by Theorem~2.3 and~Proposition~2.6 of \cite{PrST}, which states that
for any $\chi \in \g^\ast$ there exists a projective $U_\chi(\g)$-module $Q_\chi^{[p]}$
such that $U_\chi(\g)$ is isomorphic to a direct sum of $p^d$ copies of $Q_\chi^{[p]}$ as a left $\g$-module, and is isomorphic to
$\Mat_{p^d} \End_\g(Q_\chi^{[p]})^\opp$ as algebras, where $d := \frac{1}{2}(\dim \g - \dim \g_\chi)$.
In the case of regular $\chi$ Corollary~\ref{firstcor} now implies that $\varphi_\chi(Z(\g)) \cong \End_\g(Q_\chi^{[p]})^\opp$.
This  is a slight generalisation of the well known fact that the restricted finite $W$-algebra associated to regular nilpotent elements is isomorphic
to the centre of the reduced enveloping algebra, which is a modular analogue of the theorem, due to Kostant, stating that the finite $W$-algebra
associated to a regular element over $\C$ is isomorphic to the centre of the universal enveloping algebra (see \cite{Wa} and the references therein).
\end{Remark}

\section{Decomposition classes}\label{deformation}

In this section we shall need to recall many general results, and so unless otherwise stated we assume that $G$ is simple, simply connected,
that the characteristic is good for $G$ (possibly zero) and that $\g$ has a non-degenerate trace form.
The decomposition of $\g$ into adjoint orbits can be better understood by grouping certain orbits together as follows.
Write $x = x_n + x_s$ for the Jordan decomposition of $x\in \g$ into nilpotent and semisimple parts (see \cite[Theorem~3.5]{FS}).
We define an equivalence relation on $\g$ by declaring that $x\sim y$ whenever there exists $g\in G$ such that $\Ad(g) x_n = y_n$ and $\Ad(g) \g_{x_s} = \g_{y_s}$.
The equivalence classes under this relation are known as decomposition classes.
Equivalently, we can describe them as the sets $\Ad(G)(x_n + \z(\g_{x_s})^\reg)$ where the superscript $\reg$ is used to denote the set of elements
therein which have a centraliser in $\g$ of minimal dimension. The latter description immediately leads to a classification of the classes by
combinatorial data. Let $\Ad(G) (\g_0, e_0)$ denote the $G$-orbit of the pair $(\g_0, e_0)$ with $G$-action diagonally,
where $\g_0 \subseteq \g$ is a Levi subalgebra and $e_0 \subseteq \g_0$ is a nilpotent element. To $\Ad(G)(\g_0, e_0)$ we can associate the
class $$\D(\g_0, e_0) := \Ad(G)(e_0 + \z(\g_0)^\reg).$$ The use of decomposition classes was
pioneered by Walter Borho \cite{Bo} as the key tool in the classification of the sheets of $\g$ over $\C$.
We recommend \cite{PS} as a reference over fields of positive characteristic.

We now recall some highlights of the theory which will be needed in the sequel. Let $\N(\g)$ denote the nilpotent cone of $\g$, ie. the set of all nilpotent elements.
By \cite[Proposition~2.5]{PS} and \cite[\S 7]{Ja2}, for each class $\D$ the intersection $\overline{\D} \bigcap \N(\g)$ is irreducible. Then by the finiteness
of the number of nilpotent orbits in $\g$ there is a unique dense orbit in the intersection, which we denote by $\O_\D$.

Recall the theory of induced nilpotent orbits due to Lusztig and Spaltenstein. If $\g_0 \subseteq \g$ is a Levi factor of some parabolic
$\p = \g_0 \bigoplus \n$ and $\O_0 \subseteq \g_0$ is a nilpotent orbit then $\Ad(G)(\O_0 + \n) \subseteq \N(\g)$ is irreducible, hence
also contains a dense nilpotent orbit which we denote $\Ind_{\g_0}^\g(\O_0)$. It is known that this orbit only depends upon $\g_0$, not on $\p$,
and so the notation is justified \cite[Remark~2.7]{PS}.
\begin{Lemma}\label{inductionandclasses} Let $e_0 \in \O_0 \subseteq \g_0$ be an element of a nilpotent orbit in a Levi subalgebra. The following hold:
\begin{enumerate}
\smallskip
\item{$\Ind_{\g_0}^{\g}(\O_0) = \O_{\D(\g_0, e_0)}$.}
\medskip
\item{$\codim_{\N(\g)} \Ind_{\g_0}^\g(\O_0) = \codim_{\N(\g_0)} \O_0.$}
\smallskip
\end{enumerate}
\end{Lemma}
\begin{proof}
Write $\D := \D(\g_0, e_0)$ and $\O := \O_\D$. Equation (3) in the proof of \cite[Proposition 2.5]{PS} implies that
$e_0 + \n \subseteq \overline{\D}$, and so $\Ind_{\g_0}^{\g}(\O_0) \subseteq \overline{\D}$.
Since $\O$ is the unique dense orbit in $\overline{\D} \bigcap \N(\g)$ the first part will follow from the
equality $\dim \O = \dim \Ind_{\g_0}^\g(\O_0)$. Proposition~2.5 and Theorem~2.8 in {\it loc. cit.} say that both of
these numbers coincide with $\dim \g - \dim (\g_0)_{e_0}$. For the second part we use the formulas
$\dim \N(\g) = \dim \g - \rank\, \g$ and $\dim \N(\g_0) = \dim \g_0 - \rank\, \g_0$ (see \cite[Theorem~6.4]{Ja2}) to calculate
\begin{eqnarray*}
\codim_{\N(\g)} \Ind_{\g_0}^\g(\O_0) & = & \dim \N(\g) - \dim \g + \dim (\g_0)_{e_0} \\
							& = & \dim \N(\g_0) - \dim \g_0 + \dim (\g_0)_{e_0}\\
							& = & \codim_{\N(\g_0)} \O_0.
							\end{eqnarray*}
\end{proof}

We shall need the following important fact.
\begin{Theorem}\label{pind}
Fix a root system $\Phi$, a sub-root-system $\Phi_0$ and a Bala--Carter label {\sf X} associated to $\Phi$. Let $G$ be the simple, simply
connected group of good characteristic $p \geq 0$. Let $\g_0\subseteq \g$ be the Levi subalgebra corresponding to $\Phi_0$ and let $\O_0$
be the nilpotent orbit with label {\sf X}. Then the Bala--Carter label of $\Ind_{\g_0}^\g(\O_0)$ does not depend upon our choice
of good characteristic $p$.
\end{Theorem}
\begin{proof}
When $\Phi$ is an exceptional root system the result follows from the first part of the previous lemma along with Theorem~1.4
from \cite{PS}. When $G$ is classical it suffices to show that the partitions associated to induced orbits do not depend upon the
choice of good characteristic. For this, observe that the deductions regarding partitions given in Theorems~7.2.3 and 7.3.3 of \cite{CM}
are valid in any characteristic provided the dimension of the corresponding orbit is also independent of $p$. This latter fact 
follows from \cite[Theorem~2.6~\&~2.7]{PrN}.
\end{proof}

It is well known that in  the Lie algebra of a simple algebraic group over a field of good characteristic there is a unique largest
non-regular nilpotent orbit, ie. one which dominates all other non-regular nilpotent orbits \cite[Theorem~5.7]{Hu}, \cite[III.3.12]{SS}, \cite{PrN}.
We denote this orbit $\Os$ and call it subregular. Recall that a rigid nilpotent orbit is one which cannot be properly induced.
\begin{Lemma}\label{inducesub}
Let $G$ be simple, and retain the assumption that $\Char(\k) = p > 2$ is very good for $G$.
For a Levi subalgebra $\g_0 \subseteq \g$ and a nilpotent orbit $\O_0\subseteq \g_0$ the following are equivalent:
\begin{enumerate}
\item{$\Os = \Ind_{\g_0}^\g(\O_0)$ and $\O_0$ is rigid;}
\item{$\g_0$ has semisimple rank 1 and $\O_0 = 0$.}
\end{enumerate}
\end{Lemma}
\begin{proof}
The subregular orbit is determined uniquely by the fact that $$\codim_{\N(\g)} \Os = 2.$$
If $\g_0\subseteq \g$ is a Levi subalgebra of semisimple rank 1 and $\O_0 = 0$, then we have $\codim_{\N(\g_0)} \O_0 = 2$
and so $\Os = \Ind_{\g_0}^\g(\O_0)$ by part 2 of Lemma~\ref{inductionandclasses}.

We prove the converse by induction. Strangely enough the inductive argument breaks down over fields of
positive characteristic however Theorem~\ref{pind} allows us to work in the situation where all groups and
algebras are defined over $\C$. When $G$ has semisimple rank 1 the statement is trivial and so suppose the rank
is $\ell > 1$. If $\O_0 \subseteq \g_0$ induces to $\Os$ then by Lemma~\ref{inductionandclasses} we have
$\codim_{\N(\g_0)} \O_0 = 2$. We can write $\g_0 \cong \g_0^1 \times\g_0^2 \times\cdots \times\g_0^k \times \z(\g_0)$
for simple factors $\g_0^i$ and we can decompose $\O_0 = \O_0^1 \times \cdots \times \O_0^k$ (here we use that the
base field is $\C$). The product $\O_0^1 \times \cdots \times \O_0^k$ is rigid in $\g_0$ if and only if each factor
is rigid in $\g_0^i$. Since nilpotent orbits have even codimension it follows that $\dim \O_0^i = \dim \N(\g_0^i)$ for
all $i$ apart from a single index $i'$, and that $\dim \O_0^{i'} = \dim \N(\g_0^{i'}) - 2$. Now we may apply the inductive
hypothesis to the pair $(\g_0^{i'}, \O_0^{i'})$ to deduce that $\O_0^{i'}$ is zero and $\g_0^{i'}$ has semisimple rank 1.
Since every $\O_0^i$ with $i \neq i'$ is both regular and rigid we conclude that $\O_0^i = 0$ and $\g_0^{i}$ has semisimple
rank 0 for all such $i$. Now the claim follows. 
\end{proof}

We continue to assume that $G$ is a simple algebraic group. We shall need to describe the
maximal non-regular decomposition classes in $\g$. Recall that
$\g \setminus \g_\reg$ is a Zariski closed set, hence it can be expressed as an irredundant union of irreducible affine varieties
$$\g \setminus \g_\reg = \bigcup_{i=1}^k X_i.$$
Clearly $\g\setminus \g_\reg$ is a union of decomposition classes, and from the finiteness of the latter
we deduce that there exist decomposition classes $\D_1,...,\D_k$ with $X_i = \overline{\D}_i$.

Now consider the collection of all Levi subalgebras of semisimple rank 1, and choose a set of representatives
$\l_1,...,\l_{k'}$ for the $G$-conjugacy classes of such Levis. Write $\O_i\subseteq \l_i$ for the zero orbit
and consider the decomposition classes $\D_i' := \D(\l_i, \O_i)$ with $i = 1,...,k'$.
\begin{Proposition}\label{decomplineup}
We have $k = k'$ and, after relabeling, $\D_i = \D_i'$ for $i=1,...,k$.
\end{Proposition}
\begin{proof}
Consider the sets $\g^{(i)} := \{ x\in \g : \dim \g_x = i\}$. They are evidently locally closed and we have
$\overline{\g^{(i)}} = \{x \in \g : \dim \g_x \geq i\}$. Since elements $x$ in the subregular nilpotent orbit satisfy
$\dim \g_x = \ell + 2$, and since $\codim_\g \g_x$ is always even, we deduce that $\g \setminus \g_\reg = \overline{\g^{(\ell + 2)}}.$
Since $\dim \g_x$ remains constant as $x$ varies over a decomposition class, it will suffice to check that the
decomposition classes $\D_1',...,\D_{k'}'$ are precisely those which:
\begin{enumerate}
\item[i)]{contain orbits of dimension $\dim \g - \ell - 2$, and;}
\item[ii)]{are not properly contained in the closure of any other such decomposition class.}
\end{enumerate}
By Lemma~\ref{inductionandclasses} the first property is satisfied for a class $\D(l, \O_\l)$ if and only if $\Ind_{\l}^\g(\O_\l)$ is the subregular
nilpotent orbit of $\g$ and, by Lemma~\ref{inducesub}, this holds for $\D_1',...,\D_{k'}'$.

The irreducible components of the sets $\g^{(i)}$ known as the sheets of $\g$ and,
according to \cite[Theorem~2.8]{PS}, the dense decomposition classes in sheets correspond precisely to pairs
$(\g_0, \O_0)$ where the orbit $\O_0$ is rigid. Applying Lemma~\ref{inducesub} once more we deduce that (ii)
holds for the classes $\D_1',...,\D_{k'}'$.
\end{proof}

The following result holds in the generality of arbitrary restricted Lie algebras.
\begin{Proposition}\label{deform}
Suppose that $X \subseteq \g^*$ is a locally closed affine variety and $d = \dim Z_\chi(\g)$ remains constant as $\chi$ varies over $X$.
Then $\dim Z_\chi(\g) \geq d$ for all $\chi \in \overline{X}$. In other words, $\dim Z_\chi(\g)$ is upper semicontinuous as a function
from $\g^\ast$ to $\Z_{\geq 0}$.
\end{Proposition}
\begin{proof}
Let $x_1,...,x_n$ be a basis for $\g$. For any $\chi \in \g^\ast$ the PBW theorem implies that $U_\chi(\g)$ has a basis
\begin{eqnarray*}\label{basiss}
\{x_1^{a_1}\cdots x_n^{a_n} : 0 \leq a_i < p \text{ for all } i\}.
\end{eqnarray*}
We can order this basis in any way, say $y_1,...,y_{N}$ where $N := p^n$. Now let $\chi \in \g^\ast$
and notice that $\ad(\g)$ annihilates the ideal $I_\chi$ of $Z_p(\g)$, hence preserves $I_\chi U(\g)$.
In this way we obtain an adjoint action $\ad_\chi : \g \rightarrow \End_\k U_\chi(\g)$. 
The centre $Z_\chi(\g)$ is the intersection of $\Ker \,\ad_\chi(x_i)$ over $i=1,...,n$.
Using the basis $y_1,...,y_N$ we can identify all reduced enveloping algebras with one single $\k$-vector
space $E$ spanned by the symbols $y_1,...,y_N$. We claim that for each $x \in \g$ the linear map
$\ad_\chi(x) \in \End_\k(E)$ depends polynomially on $\chi$.

To see the claim, write $\tilde{y}_1,...,\ty_N$ for a choice of representatives for $y_1,...,y_N$ in $U(\g)$,
ie. $\ty_i + I_\chi U(\g) = y_i$ for each $i$. Now introduce $n$ 
variables $T_{i}$ with $1\leq i \leq n$, define $\k[T] := \k[T_1,...,T_n]$ and consider $U(\g)[T] := U(\g) \bigotimes_{\k} \k[T]
\cong U(\g[T])$, where $\g[T] := \g\bigotimes_\k \k[T]$ with the obvious Lie algebra structure.
We may view $\g[T]$ as a restricted $\k[T]$-Lie algebra with $p$-centre $Z_p(\g[T])$ generated
by $1\bigotimes \k[T]$ and $Z_p(\g) \bigotimes 1$. Consider the ideal $I_T$ in $Z_p(\g[T])$ generated by expressions
$x_i^p - x_i^{[p]} - T_i^p$ with $i=1,...,n$. The PBW theorem for reduced enveloping algebras shows that a
$\k[T]$-basis for $U(\g[T])/I_T U(\g[T])$ is given by the images in the quotient of expressions $\ty_i \otimes 1$ with $i=1,...,N$ and
so a $\k$-basis is given by $\ty_i \otimes T^{\bi}$, with $i=1,...,N$
and $\bi = (i_{j})_{1\leq j\leq n} \in \Z_{\geq 0}^n$ and $T^\bi := \prod_{j} T_{j}^{i_{j}}$.
Now consider $u \in U(\g) \subseteq U(\g[T])$ and write
\begin{eqnarray*}\label{specialise}
u + I_T U(\g[T]) = \sum_{i=1}^N c_i \ty_i \otimes f_i(T) + I_T U(\g[T])
\end{eqnarray*}
\noindent where $f_i(T) \in \k[T]$. If we specialise $T_i \mapsto \chi(x_i)$ then
$U(\g[T]) \mapsto U(\g)$, $I_T \mapsto I_\chi$ and $u \mapsto u$. Therefore
\begin{eqnarray}\label{specialise2}
u + I_\chi U_\chi(\g) = \sum_{i=1}^n c_i f_i(\chi (x_1),...,\chi(x_n)) y_i \in U_\chi(\g).
\end{eqnarray}
\noindent Now we complete the proof of the claim. Fix $i = 1,...,n$ and $j = 1,...,N$ and consider
\begin{eqnarray}\label{specialise3}
\ad_\chi(x_i)y_j = \ad(x_i) \ty_j + I_\chi U(\g).
\end{eqnarray}
\noindent Setting $u = \ad(x_i) \ty_j$ in equation (\ref{specialise2}), we see that the right hand side of (\ref{specialise3}) is a linear combination 
of $y_1,...,y_N$ and the coefficients are polynomials in $\chi$, as claimed.

Now we would like to express $Z_\chi(\g)$ as the kernel of a single matrix. Consider the diagonal embedding
$\Delta : E \hookrightarrow \bigoplus_{i=1}^n E$, the direct sum $\oplus_{i=1}^n \ad_\chi(x_i) \in \End_\k(\bigoplus_{i=1}^n E)$ and the map
$$s(\chi) := (\oplus_{i=1}^n \ad_\chi(x_i))|_{\Delta(E)} \in \Hom_\k(E, \bigoplus_{i=1}^n E).$$
Under the identification between $E$ and $U_\chi(\g)$ the kernel of $s(\chi)$ identifies with the intersection
$$\bigcap_{i=1}^n \Ker\, \ad_\chi(x_i) = Z_\chi(\g).$$
\noindent By the above, $s$ may be viewed as a polynomial map $$\g^\ast\longrightarrow \Hom_\k(E, \bigoplus_{i=1}^n E).$$ In particular, it is
Zariski continuous and the preimages of locally closed sets are locally closed.
Now observe that for $k \in \Z_{\geq 0}$ the sets $$Y_k := \{ H \in \Hom_\k(E, \bigoplus_{i=1}^n E) : \dim \Ker H = k\}$$ are locally
closed subspaces such that $Y_j \subseteq \overline{Y_k}$ if and only if $j \geq k$. This is because, viewing $\Hom_\k(E, \bigoplus_{i=1}^n E)$
as a vector space of matrices, the sets $\bigcup_{k \geq k'} Y_k$ are defined by the vanishing of certain minors depending on $k'$.
But now
$$L_k:= \{\chi \in \g^\ast : \dim Z_\chi(\g) = k\} = s^{-1}(Y_k)$$
and we have shown that the level sets $L_k$ of $\chi \mapsto \dim Z_\chi(\g)$ are locally closed subsets of $\g^\ast$, and
that $L_j \subseteq \overline{L_k}$ if and only if $j \geq k$.

Finally take $X$ as described in the statement of the Proposition. By assumption $X\subseteq L_d$. By the properties we have deduced about $L_d$,
we know that if $\chi \in \overline{X}$ then $\chi \in L_j$ for some $j \geq d$, which completes the proof.
\end{proof}

\section{The proof of the main theorem}\label{proof}

Here we complete the proof of the main theorem, which states that the natural map $\varphi_\chi : Z(\g) \rightarrow Z_\chi(\g)$
is surjective if and only if $\chi$ is regular.
\begin{proof}
By Corollary~\ref{regularcase} we only need to show that if $\dim \g_\chi \neq \ell$ then $\varphi_\chi$ is not surjective.
The argument proceeds by induction on the reductive rank of $G$. The base case is $\rank(G) = 1$, which implies $G = SL_2$ or $\mathbb{G}_m$.
Here there are at most two types of orbits, the zero orbit and the regular orbits, which are
dealt with by Theorem~\ref{zerocase} and Corollary~\ref{regularcase}.

Now we may suppose that $\rank(G) = \ell > 1$ and that the theorem holds for all groups which satisfy the hypotheses
and have reductive rank less than $\ell$. We begin by observing that there is a reduction to the case where $G$ is simple.
Suppose that $\g = \z(\g) \oplus (\bigoplus_{i=1}^k \g_i)$ where $\g_i$ are simple restricted ideals. 
Then we have $$U_\chi(\g) \cong U_{\chi}(\z(\g)) \otimes (\bigotimes_{i=1}^k U_{\chi}(\g_i)).$$
Now the map $\varphi_\chi$ is just the tensor product of the maps corresponding to each simple factor or to $\z(\g)$,
which is surjective if and only if each of the tensor factors is so. Note that the map corresponding to $\z(\g)$ is always surjective.
Furthermore, $\chi$ is regular in $\g^\ast$ if and only if each $\chi_i$ is regular in $\g_i^\ast$ for $i = 1,...,k$.
Hence by induction the theorem holds for all reductive algebraic groups for which $\g$ contains a proper ideal.
Since $\g = \z(\g)\bigoplus [\g,\g]$ in very good characteristic, it only remains to treat the case where $G$ is simple.

Let $\l_1,...,\l_k$ be a set of representatives for the $\Ad(G)$-conjugacy classes of Levi subalgebras of $\g$ which have
semisimple rank 1.
Let $\O_i \subseteq \l_i$ denote the zero orbit. We consider the decomposition classes $\D_i := \D(\l_i, \O_i)$
with $i = 1,...,k$. Let $x \in \D_i$ for some $1\leq i \leq k$ and let $\kappa(x) = \chi = \chi_s + \chi_n \in \g^*$
be the Jordan decomposition. By definition $\g_{\chi_s}$ is a $G$-conjugate of $\l_i$. From Corollary~\ref{dimscorol}
it follows that $\dim Z_{\chi}(\g) = \dim Z_{0}(\l_i)$ and the latter is $> p^\ell$ by Theorem~\ref{zerocase}.

According to Proposition~\ref{decomplineup} the set $\bigcup_{i=1}^k \D_i$ is dense inside $\g \setminus \g_\reg$. In other words
$x \in \g$ is not regular if and only if $x$ lies in the closure of $\bigcup_{i=1}^k \D_i.$ We have shown that
$\dim Z_{\chi}(\g) > p^\ell$ for all elements of $\kappa(\bigcup_{i=1}^k \D_i)$ and so by Proposition~\ref{deform}
we see that the same holds for all elements of $\g^\ast$ which are not regular. By Corollary~\ref{firstcor} we deduce
that $\varphi_\chi$ is not surjective for all such elements, which completes the proof.

\end{proof}

\end{document}